\documentclass[a4paper, 11pt]{article}

\usepackage{amsmath}
\usepackage{amsfonts}
\usepackage{amssymb}
\usepackage[english]{babel}
\usepackage{graphicx}
\usepackage{amsthm}
\usepackage{mathrsfs}
\usepackage{pgf,tikz}
\usepackage[normalem]{ulem}
\usepackage{cancel}
\usepackage{amstext} 
\usepackage{array}   
\usepackage{caption}
\newcolumntype{C}{>{$}c<{$}} 
\definecolor{uququq}{rgb}{0.25,0.25,0.25}
\newtheorem{thm}{Theorem}[section]
\newtheorem{cor}[thm]{Corollary}
\newtheorem{lem}[thm]{Lemma}
\newtheorem{prop}[thm]{Proposition}
\newtheorem{constr}[thm]{Construction}

\usetikzlibrary{arrows}
\theoremstyle{definition}
\newtheorem{defn}[thm]{Definition}
\theoremstyle{definition}
\newtheorem{fact}[thm]{Fact}
\newtheorem{rem}[thm]{Remark}
\theoremstyle{definition}
\newtheorem{ex}[thm]{Example}
\newcommand{\N}{\mathbb{N}}
\newcommand{\Z}{\mathbb{Z}}

\newcommand{\F}{\mathbb{F}}

\newcommand{\comment}[1]{}
\numberwithin{equation}{section}

\begin{document}
\date{}
\title{Super-regular Steiner 2-designs}
\author{Marco Buratti \thanks{Dipartimento di Matematica e Informatica, Universit\`a di Perugia, via Vanvitelli 1, Italy. Email: buratti@dmi.unipg.it}\\
\\
Anamari Naki\'c \thanks{Faculty of Electrical Engineering and Computing,
University at Zagreb, Croatia, email: anamari.nakic@fer.hr}}
\date{\today}
\maketitle
\begin{abstract}
A design is additive under an abelian group $G$ (briefly, $G$-additive) 
if, up to isomorphism, its point set is contained in $G$ and the elements of each block sum up to zero.
The only known Steiner 2-designs that are $G$-additive for some $G$ have block size which is either a
prime power or a prime power plus one. Indeed they are the point-line
designs of the affine spaces $AG(n,q)$, the point-line designs of the projective planes 
$PG(2,q)$, the point-line designs of the projective spaces $PG(n,2)$ and a
sporadic example of a 2-(8191,7,1) design.  
In the attempt to find new examples, possibly with a block size which is 
neither a prime power nor a prime power plus one, we look for Steiner 
2-designs which are strictly $G$-additive (the point set is exactly $G$) and $G$-regular
(any translate of any block is a block as well) at the same time.
These designs will be called ``$G$-super-regular". Our main result is that 
there are infinitely many values of $v$ for which there exists a
super-regular, and therefore additive, $2$-$(v,k,1)$ design
whenever $k$ is neither singly even nor of the form $2^n3\geq12$.
The case $k\equiv2$ (mod 4) is a genuine exception whereas
$k=2^n3\geq12$ is at the moment a possible exception.
We also find super-regular $2$-$(p^n,p,1)$ designs with $p\in\{5,7\}$ and $n\geq3$
which are not isomorphic to the point-line design of $AG(n,p)$.
\end{abstract}


\normalsize
\eject
\section{Introduction}
We recall that a $t$-$(v,k,\lambda)$ design is a pair ${\cal D}=(V,{\cal B})$ with $V$ a set of $v$ {\it points} and ${\cal B}$
a collection of $k$-subsets of $V$, called {\it blocks}, such that any $t$-subset of $V$ is
contained in exactly $\lambda$ blocks. 
It is understood that $1\leq t\leq k\leq v$. The design is said to be {\it simple} if $\cal B$ is a set, i.e., if
it does not have repeated blocks. When $v=k$ we necessarily have only one block,
coincident with the whole point set $V$, repeated $\lambda$ times; in this case the
design is said to be {\it trivial}.

In the important case of $\lambda=1$ one speaks of a {\it Steiner $t$-design} and the notation $S(t,k,v)$
is often used in place of ``$t$-$(v,k,1)$ design".
An isomorphism between two designs $(V,{\cal B})$ and $(V',{\cal B}')$
is a bijection $f: V \longrightarrow V'$ turning ${\cal B}$ into ${\cal B}'$.
Of course the study of $t$-designs is done up to isomorphism.

An {\it automorphism group} of a design ${\cal D}=(V,{\cal B})$ is a group $G$ of permutations on $V$
leaving ${\cal B}$ invariant, i.e., a group of isomorphisms of $\cal D$ with itself. 
If $G$ acts regularly -- i.e., sharply transitively -- on the points, then
$\cal D$ is said to be {\it regular} under $G$ (briefly {\it $G$-regular}). Up to isomorphism, 
a $G$-regular design has point set $G$ and any translate $B+g$ of any block $B$ is a block as well.

For general background on $t$-designs we refer to \cite{BJL} and \cite{CD}.

Throughout the paper, every group will be assumed finite and abelian unless specified otherwise.
A subset $B$ of a group $G$ will be said {\it zero-sum} if its elements sum up to zero.
Representing the blocks of a design as zero-sum subsets 
of a commutative group turned out to provide an effective algebraic tool for studying their automorphisms 
(see, e.g., Example 3.7 in \cite{CFP2}).
Also, some recent literature provides remarkable examples of usage of zero-sum blocks in the construction 
of combinatorial designs (see, e.g., \cite{BCHW, K}).
This gives even more value to the  interesting theory on {\it additive designs} introduced in
\cite{CFP} by Caggegi, Falcone and Pavone.
Other papers on the same subject by some of these authors are \cite{C, CF,FP,Pavone}.
They say that a design is additive if it is embeddable into an abelian group 
in such a way that the sum of the elements in any block is zero. 

We reformulate just a little bit the terminology as follows.
\begin{defn}
A design $(V,{\cal B})$ is {\it additive under an abelian group $G$} (or briefly {\it $G$-additive}) 
if, up to isomorphism, we have:
\begin{itemize}
\item[$(1)$] $V\subset G$;
\item[$(2)$] $B$ is zero-sum $\forall B\in{\cal B}$.
\end{itemize}
If in place of condition (2) we have the much stronger condition
 \begin{itemize}
\item[$(2)_s$] ${\cal B}$ is precisely the set of all zero-sum $k$-subsets of $V$,
\end{itemize}
then the design is  {\it strongly $G$-additive}.
\end{defn}

By saying that a design is additive (resp., strongly additive)  we will mean that it
is $G$-additive (resp., strongly $G$-additive) for at least one abelian group $G$.
Note that we may have designs which are $G$-additive and\break $H$-additive at the same time even though
none of them is isomorphic to a subgroup of the other. 
For instance, it is proved in \cite{CFP} that if $p$ is a prime, then the point-line design of the affine
plane over $\Z_p$, which is obviously $\Z_p^2$-additive, is also strongly $\Z_p^{p(p-1)/2}$-additive. 

In general, to establish whether an additive design is also strongly additive appears to be
hard. Examples of additive $2$-$(v,k,\lambda)$ designs which are not strongly additive 
are given for $2\leq \lambda\leq 6$ in \cite{Pavone}. The question on whether 
there exists an additive Steiner 2-design which is not strongly additive is still open.

We propose to consider the $G$-additive designs whose set of points is precisely $G$ or $G\setminus\{0\}$.

\begin{defn}
An additive design is {\it strictly $G$-additive} or {\it almost strictly\break $G$-additive} 
if its point set is precisely $G$ or $G\setminus\{0\}$, respectively.
\end{defn}

Of course strictly additive (resp., almost strictly additive) means strictly $G$-additive  (resp., almost strictly $G$-additive)
for a suitable $G$.
As it is standard, $\F_q$ will denote the field of order $q$ and also, by abuse of notation, its additive group.
It is quite evident that the  $2$-$(q^n,q,1)$ design of points and lines of $AG(n,q)$ 
(the $n$-dimensional affine geometry over $\F_q$) is strictly $\F_{q^{n}}$-additive.

As observed in \cite{CF}, every $2$-$(2^v-1,2^k-1,\lambda)$ design over $\F_2$ is almost strictly 
$\F_{2^v}$-additive\footnote{A 2-design is {\it over $\F_q$} if its points are those of a 
projective geometry over $\F_q$ and the blocks are suitable subspaces of this geometry.}.
Thus there exists an almost strictly $\Z_2^v$-additive\break $2$-$(2^v-1,7,7)$ design for any odd $v\geq3$ in view of the main results in \cite{BN,Thomas}.
Also, there is an almost strictly $\Z_2^{v+1}$-additive $2$-$(2^{v+1}-1,3,1)$ design that is the point-line design of $PG(v,2)$ 
(the $v$-dimensional projective geometry over $\F_2$).
Finally, each of the well-celebrated designs found in \cite{BEOVW} and revisited in \cite{BNW}
is an almost strictly $\Z_2^{13}$-additive $2$-$(8191,7,1)$ design.

Almost all known additive designs have quite large values of $\lambda$.
For instance, it is proved in \cite{Pavone3} that if $p$ is an odd prime and $k=mp$ does not exceed $p^n$, 
then all zero-sum $k$-subsets of $\F_{p^n}$ form the block-set of a strongly additive 
$2$-$(p^n,k,\lambda)$ design with $\lambda={1\over p^n}{p^n-2\choose k-2}+{k-1\over p^n}{p^{n-1}-1\choose m-1}$. 
Applying this with $p=3$, $n=4$ and $k=6$, one finds a strongly additive  $2$-$(81,6,18551)$ design.

A sporadic example with $\lambda=2$ is the strictly $\Z_3^4$-additive $2$-$(81,6,2)$ design 
given in \cite{N} and some more classes with a relatively small $\lambda$ will be given in \cite{BN2}.
Anyway, what is most striking is the shortage of additive Steiner 2-designs. Up to now, only 
three classes were known: 
\begin{itemize}
\item[C1.] the designs of points and lines of the affine geometries over
any field $\F_q$ (which are strictly additive); 
\item[C2.] the designs of points and lines of the projective 
geometries over $\F_2$ (which are almost strictly additive); 
\item[C3.] the designs of points and lines
of the projective planes over any field $\F_q$ (which are strongly additive under a ``big" group \cite{CFP}).   
\end{itemize}
Nothing else was known, except for the sporadic
example of the $2$-$(8191,7,1)$ design mentioned above.

Hence to find additive Steiner 2-designs with new parameters,
in particular with block size which is neither a prime power nor a prime power plus one,
appears to be challenging. 

Note that the $2$-$(q^n,q,1)$ designs mentioned above are also $\F_{q^n}$-regular.
This fact suggests that a natural approach for reaching our target is to look for strictly $G$-additive Steiner 2-designs 
which are also $G$-regular. Let us give a name to the
designs with these properties.
\begin{defn}
A design is {\it super-regular} under an abelian group $G$ (or briefly {\it $G$-super-regular})
if it is $G$-regular and strictly $G$-additive at the same time. 
\end{defn}

Similarly as above, super-regular will mean $G$-super-regular for a suitable $G$. Super-regular Steiner
2-designs will be the central topic of this paper. Our main result will be the following. 
\begin{thm}\label{main}
Given $k\geq3$, there are infinitely many values of $v$ for which there
exists a super-regular $2$-$(v,k,1)$ design with the genuine exceptions of the singly even values
of $k$ and the possible exceptions of all $k=2^n3\geq12$.
\end{thm}

As an immediate consequence, we have the existence of a strictly additive Steiner 2-design
with block size $k$ for any $k$ with the same exceptions as in the above statement.

A major disappointment is that the smallest $v$ for which, fixed $k$, we are able to say that a super-regular 
$2$-$(v,k,1)$ design exists, is huge. Suffice it to say that for $k=15$ this value is $3\cdot5^{31}$.
Consider, however, that there are several asymptotic results proving the existence of some designs 
as soon as the number of points is admissible and greater than a bound which is not even quantified. 
This happens, for instance, in the outstanding achievement 
by P. Keevash \cite{K} on the existence of Steiner $t$-designs. Usually, these asymptotic results
are obtained  via probabilistic methods and are not constructive. Our methods are algebraic and
``half constructive". We actually give a complete recipe for building a super-regular $2$-$(kq,k,1)$ design
under $G\times\F_q$ (with $G$ a suitable group of order $k$) whenever $q$ is an admissible power of a 
prime divisor of $k$ sufficiently large. Yet, in building 
every {\it base block} we have to pick the second coordinates of its elements, one by one, in a way 
that suitable cyclotomic conditions are satisfied and these choices are
not ``concrete"; they are realizable only  in view of some theoretical arguments deriving from the 
{\it theorem of Weil on multiplicative character sums}.

In the penultimate section we will have a look at the super-regular non-Steiner 2-designs.

The paper will be organized as follows. In the next section we first prove two elementary necessary
conditions for the existence of a strictly\break $G$-additive $2$-$(v,k,1)$ design: $G$ cannot have 
exactly one involution, and every prime factor of $v$ must divide $k$. 

In Section 3 we recall some basic facts on regular designs and show that any super-regular design
can be completely described in terms of differences.
In particular, we prove that a sufficient condition for the existence of a\break
$(G\times\F_q)$-super-regular design with $G$ a non-binary group of order $k$
and $q$ a power of a prime divisor of $k$ is the existence of an additive 
$(G\times\F_q,G\times\{0\},k,1)$ difference family.
This is a set ${\cal F}$ of zero-sum $k$-subsets of $G\times \F_q$ whose list of differences is $(G\times\F_q)\setminus(G\times\{0\})$.
In Section 4 we prove that such an $\cal F$ cannot exist for $k=2^n3\geq12$, clarifying in this way
why this case is so hard. 

In Sections 5 it is shown that a difference family as above can be realized by 
suitably lifting the blocks of an additive $(G,k,\lambda)$ strong difference family, that is 
a collection of zero-sum $k$-multisets on $G$ whose list of differences is $\lambda$ times $G$.

In Section 6, as a first application of the method of strong difference families,
we construct a $\F_{p^n}$-super-regular $2$-$(p^n,p,1)$ design not isomorphic 
to the point-line design of $AG(n,p)$ for $p\in\{5,7\}$ and every integer $n\geq3$.

In Section 7 a combined use of strong difference families and cyclotomy leads
to a very technical asymptotic result.
As a consequence of this result, the crucial ingredient
for proving the main theorem is an additive $(G,k,\lambda)$ strong difference family
with $G$ a non-binary group of order $k$ and $\gcd(k,\lambda)=1$.

In Section 8 this ingredient is finally obtained, also via difference matrices, for all 
the relevant values of $k$ and then the main theorem is proved.

As mentioned above, the final construction leads to super-regular 
Steiner 2-designs with a huge number of points. In Section 9 it is shown that when $k=15$ 
the smallest $v$ given by this construction is $3\cdot5^{9565939}$. On the other hand we 
also show that a clever use of strong difference families and cyclotomy allows to obtain 
smaller values of $v$. Still in the case $k=15$, we first obtain $v=3\cdot5^{187}$ and then 
$v=3\cdot5^{31}$ by means of two variations of the main construction. We also suggest a 
possible attempt to obtain $v=3\cdot5^7$ by means of a computer search.

In Section 10 we sketch how the same tools used with so much labor 
to construct ``huge" super-regular Steiner 2-designs allow to rapidly obtain super-regular non-Steiner 2-designs
with a ``reasonably small" $v$ at the expense of a possibly large $\lambda$
and the loss of simplicity (each of them has ${v\over k}$ blocks repeated $\lambda$ times).
For instance, we will show the existence of a non-simple super-regular 2-design with
block size $15$ having only $3\cdot5^3$ points and with $\lambda=21$.

In the last section we list some open questions.

\section{Elementary facts about strictly additive Steiner 2-designs}

In these preliminaries we establish some constraints on the parameters
of a strictly additive Steiner 2-design. First, it is useful to show two very elementary 
facts which we believe are folklore.

\begin{fact}\label{fact1} Every non-trivial subgroup of $\F_q^*$ (the multiplicative group of $\F_q$) is zero-sum.
\end{fact}
\begin{proof}
Let $B\neq\{1\}$ be a subgroup of $\F_q^*$ and let $n$ be its order. Then, if $b$ is a 
generator of $B$, we have $b^n-1=0$, i.e., $(b-1)(\sum_{i=0}^{n-1}b^i)=0$ in $\F_q$.
Thus $\sum_{i=0}^{n-1}b^i$, which is the sum of all elements of $B$, is equal to zero.
\end{proof}

The subgroup of an abelian group $G$ consisting of
all the involutions of $G$ and zero will be denoted by $I(G)$, i.e., $I(G)=\{g\in G : 2g=0\}$.
We say that $G$ is {\it binary} when $I(G)$ has order 2, i.e., when $G$ has
exactly one involution.

\begin{fact}\label{fact2} An abelian group $G$ is not zero-sum if and only if it is binary.
\end{fact}
\begin{proof}
The elements of $G\setminus I(G)$ are partitionable into 2-subsets 
consisting of opposite elements $\{g,-g\}$ so that $G\setminus I(G)$ is zero-sum. Then 
the sum of all elements of $G$ is equal to the sum of all elements of $I(G)$. 
Now note that either $I(G)=\{0\}$ or $I(G)$ is isomorphic to $\Z_2^n$
for some $n$. If $n=1$, then $G$ is binary and the
sum of all elements of $G$ is the non-zero element of $I(G)$, that is the only involution of $G$. 
If $n>1$, then $G$ is not binary and $I(G)\setminus\{0\}$ can be viewed as the multiplicative 
group of $\F_{2^n}^*$, hence it is zero-sum by Fact \ref{fact1}.
\end{proof}

From the above fact we immediately establish when the trivial $S(2,k,k)$ is strictly additive.
\begin{prop}\label{trivial}
The trivial $2$-$(k,k,1)$ design is strictly additive
if and only if $k\not\equiv2$ $($mod $4)$.
\end{prop}
\begin{proof}
It is evident that the trivial $2$-$(k,k,1)$ design is strictly additive if and only if 
there exists an abelian zero-sum group of order $k$. Then we get the assertion 
from Fact \ref{fact2} and the following observations.

Every group of odd order $k$ is not binary.

Every group of singly even order $k$ is binary.

Among the groups of doubly even order $k$ we have 
$G=\Z_2^2\times \Z_{k/4}$ which is not binary.
\end{proof}

We recall that the {\it radical} of an integer $n$, denoted by $rad(n)$, is the product of all prime
factors of $n$. Thus, the fact that a finite field $\F_q$ has characteristic $p$ can be also expressed
by saying that $rad(q)=p$. The following property reduces significantly the admissible parameters 
for a strictly additive $2$-$(v,k,1)$ design. 

\begin{prop}\label{rad}
If a strictly $G$-additive $2$-$(v,k,1)$ design exists, then $G$ is zero-sum and
the radical of $v$ is a divisor of $k$.
\end{prop}
\begin{proof}
Let ${\cal D}=(G,{\cal B})$ be a $2$-$(v,k,1)$ design which is strictly additive under $G$.
For any fixed element $g$ of $G$, let ${\cal B}_g$ be the set of blocks through $g$ and 
recall that its size $r$ (the so-called {\it replication number} of $\cal D$) does not depend on $g$.
Now consider the double sum $$\sigma_g=\sum_{B\in{\cal B}_g}(\sum_{b\in B}b).$$
We have $\sum_{b\in B}b=0$ for every $B\in{\cal B}_g$ because $\cal D$ is strictly additive,
hence $\sigma_g$ is null. Also note that in the expansion of $\sigma_g$ the fixed element $g$ appears
as an addend exactly $r$ times whereas any other element $h$ of $G$ appears as an addend exactly once.
Thus $\sigma_g$ can be also expressed as $(r-1)g+\sum_{h\in G}h$.
We conclude that we have
$$(r-1)g+\sum_{h\in G}h=0 \quad \forall g\in G.$$
Specializing this to the case $g=0$ we get $\sum_{h\in G}h=0$
which means that $G$ is zero-sum. Hence the first assertion is proved and we can write
$$(r-1)g=0 \quad \forall g\in G.$$
This means that the order of every element of $G$ is a divisor of $r-1$.
Let $p$ be a prime divisor of $v$, set $v=pw$, and
take an element $g$ of $G$ of order $p$ (which exists by the theorem of Cauchy). 
For what we said, $p$ divides $r-1$. Now recall that $r={v-1\over k-1}$, hence $r-1={v-k\over k-1}$. Thus 
we can write ${pw-k\over k-1}=pn$ for some integer $n$
which gives $pw-k=(k-1)pn$. This equality implies that $p$ divides $k$. 
Thus every prime factor of $v$ is also a prime factor of $k$ and the second assertion follows.
\end{proof}

In particular, considering that every abelian group of singly even order is binary,
we can state the following.

\begin{cor}\label{v singly even}
A strictly additive $2$-$(v,k,1)$ design with $v$ singly even cannot exist.
\end{cor}

In the next section we will see that in a super-regular $2$-$(v,k,1)$ design the radicals of $v$ and $k$ are even equal.

\section{Difference families}

We need to recall some
classic results on regular designs.

The {\it list of differences} of a subset $B$ of a group $G$ is the 
multiset $\Delta B$ of all possible differences $x-y$ with $(x,y)$ an ordered pair of
distinct elements of $B$. More generally, if ${\cal F}$ is a set of subsets
of $G$, the list of differences of $\cal F$ is the multiset union $\Delta{\cal F}=\biguplus_{B\in {\cal F}}\Delta B$.

Let $H$ be a subgroup of a group $G$.
A set ${\cal F}$ of $k$-subsets of $G$ is a $(G,H,k,1)$ difference family (briefly DF) 
if $\Delta{\cal F}=G\setminus H$. 

The members of such a DF are called {\it base blocks} and their number is clearly
equal to ${v-h\over k(k-1)}$ where $v$ and $h$ are the orders of $G$ and $H$, respectively. 
Thus a necessary condition for its existence is that $v-h$ is divisible by $k(k-1)$. 
It is also necessary that $I(G)$ is a subgroup of $H$ since in a list of 
differences every involution necessarily appears an even number of times.

If $G$ has order $v$ and $H=\{0\}$, one usually speaks of an {\it ordinary} $(v,k,1)$-DF in $G$. 
Instead, when $|H|=h>1$ one speaks of a $(v,h,k,1)$-DF in $G$ {\it relative to $H$} or,
more briefly, of a {\it relative} $(v,h,k,1)$-DF.

For general background on difference families as above we refer to \cite{BJL,CD}.

More generally, one can speak of a {\it difference family relative to a partial spread of $G$}, that is 
a notion introduced by the first author in \cite{JSPI}. 
A {\it partial spread} of a group $G$ is a set $\cal H$ of subgroups of $G$ whose mutual intersections are trivial.
It is a {\it spread} of $G$ when the union of its members is the whole $G$.
Also, it is said {\it of type $\tau$} to express that the multiset of the orders of its members is $\tau$.
In particular, to say that $\cal H$ is of type $\{k^s\}$ means that $\cal H$ has exactly $s$ 
members and all of them have order $k$. 

Given a partial spread $\cal H$ of a group $G$, a set $\cal F$ of $k$-subsets of $G$
is said to be a $(G,{\cal H},k,1)$ difference family if $\Delta{\cal F}$ is the set of all elements 
of $G$ not belonging to any member of $\cal H$. 
If $G$ has order $v$ and $\cal H$ is of type $\tau$,
one also speaks of a $(v,\tau,k,1)$-DF in $G$ relative to $\cal H$.
If $\tau=\{k^s\}$ for some $s$, the obvious necessary conditions 
for its existence are the following: 
\begin{equation}\label{necessaryspread}
k \ | \ v;\quad {v\over k}\equiv 1 \ ({\rm mod}  \ k-1);\quad s\equiv1 \ ({\rm mod} \ k);\quad I(G)\subset\bigcup_{H\in{\cal H}}H
\end{equation}

Clearly, a $(G,H,k,1)$-DF can be seen as a difference family relative to a partial spread 
of size 1. 

The following theorem is a special case of a general result concerning regular
{\it linear spaces} \cite{JSPI}.

\begin{thm}\label{regular}
Let $G$ be an abelian group of order $v$.
A $G$-regular $2$-$(v,k,1)$ design may exist only 
for $v\equiv1$ or $k$ $($mod $k(k-1))$  and it is equivalent to:
\begin{itemize}
\item an ordinary $(v,k,1)$-DF in $G$ when $v\equiv1$ $($mod $k(k-1))$;
\item a $(v,\{k^s\},k,1)$-DF in $G$ for some $s$ when $v\equiv k$ $($mod $k(k-1))$.
\end{itemize}
\end{thm}
We remark that the above theorem is false when $G$ is non-abelian.
\begin{rem}\label{rem}
It is useful to recall the constructive part of the proof of the above theorem
(which also works when $G$ is not abelian).
\begin{itemize}
\item[(r1)] The set of all the translates of the base blocks of an ordinary $(v,k,1)$-DF in $G$ form the block-set of 
a $G$-regular $2$-$(v,k,1)$ design.
\item[(r2)] If ${\cal F}$ is a $(v,\{k^s\},k,1)$-DF in $G$ relative to $\cal H$,
then the set of all the translates of the base blocks of $\cal F$ together 
with all the right cosets of all the members of $\cal H$ form the block-set of a $G$-regular $2$-$(v,k,1)$ design.
\end{itemize}
\end{rem}

It is immediate from Theorem \ref{regular} that any
$G$-super-regular $2$-$(v,k,1)$ design is generated by a suitable difference family.
Let us see some other consequences.

\begin{prop}\label{k=4n+2}
If there exists a $G$-super-regular $2$-$(v,k,1)$ design, then
we have:
\begin{itemize}
\item[(i)] the order of every element of $G$ is a divisor of $k$;
\item[(ii)] $v\equiv k$ $($mod $k(k-1))$;
\item[(iii)] $rad(v)=rad(k)$;
\item[(iv)] $k$ is not singly even.
\end{itemize}
\end{prop}
\begin{proof} Let ${\cal D}$ be a $G$-super-regular $2$-$(v,k,1)$ design.

(i).\quad Take any element $g$ of $G$ and any block $B$ of ${\cal D}$. By definition
of a $G$-regular design $B+g$ is a block of ${\cal D}$ as well. Also, by definition
of strictly $G$-additive design both $B$ and $B+g$ are zero-sum.
Thus, considering that the elements of $B+g$ sum up to $(\sum_{b\in B} b)+kg$, we deduce
that $kg=0$, i.e., the order of $g$ divides $k$.

(ii).\quad If $v\equiv1$ (mod $k(k-1))$, then $k$ divides $v-1$. 
By (i), the order of any $g\in G$ divides $k$, hence it also divides $v-1$.
On the other hand $ord(g)$ divides $v$ by Lagrange's theorem.
Thus $ord(g)$ would be a common divisor of $v$ and $v-1$ whichever
is $g\in G$. This would imply $v=1$ which is absurd. 
We conclude, by Theorem \ref{regular}, that we have $v\equiv k$ (mod $k(k-1))$.

(iii).\quad We already know from Proposition \ref{rad} that $rad(v)$ divides $rad(k)$. 
On the other hand $k$ divides $v$ because of condition (ii) proved above, 
hence $rad(k)$ divides $rad(v)$ and the assertion follows.

(iv).\quad $\cal D$ has at least one block $B$ 
which is a subgroup of $G$ of order $k$ in view of (ii) and Remark \ref{rem}(r2). 
Considering that $B$ is zero-sum  by assumption, the group $B$ is not binary by Fact \ref{fact2}, 
hence $k\not\equiv2$ (mod 4).
\end{proof}

Note that condition (i) of the above lemma implies, in particular, that if $p$ is a prime 
factor of $k$ but $p^2$ does not divide $k$, then the Sylow $p$-subgroup of $G$ is 
elementary abelian. Hence,
when $k$ is square-free, $G$ is necessarily a direct product of elementary abelian groups. 

In the following a $(G,{\cal H},k,1)$-DF will be said {\it additive} if all its base blocks are zero-sum and all the 
members of $\cal H$ are zero-sum (i.e., not binary) as well.

The above results (Theorem \ref{regular}, Remark \ref{rem} and Proposition \ref{k=4n+2}) allow us to state the following.
\begin{lem}\label{additiveDF}
There exists a  $G$-super-regular $2$-$(v,k,1)$ design if and only if $G$ satisfies conditions (i), (ii) of
Proposition \ref{k=4n+2} and there
exists an additive $(G,{\cal H},k,1)$-DF of type $\{k^s\}$ for some $s$.
\end{lem}

The next lemma will be our main tool to construct super-regular Steiner 2-designs.
\begin{lem}\label{additive(kq,k,k,1)}
Let $G$ be a zero-sum group of order $k$ and let $q\equiv1$ $($mod $k-1)$ be a power of a prime divisor $p$ of $k$.
If there exists an additive $(G\times\F_q,G\times\{0\},k,1)$-DF, then there exists  a $(G\times\F_{q^n})$-super-regular
$2$-$(kq^n,k,1)$ design for every $n\geq1$.
\end{lem}
\begin{proof}
The hypotheses easily imply that $G\times\F_{q^n}$ satisfies conditions (i), (ii) of
Proposition \ref{k=4n+2} for every $n$.  
Let $\cal F$ be an additive $(G\times\F_q,G\times\{0\},k,1)$-DF, 
and let $S$ be a complete system of representatives for
the cosets of $\F_q^*$ in $\F_{q^n}^*$. For every base block $B$ of $\cal F$
and every $s\in S$, let $B\circ s$ be the subset of $\F_{q^n}^*$ obtained from $B$ by 
multiplying the second coordinates of all its elements by $s$. It is easy to see that 
${\cal F}\circ S:=\{B\circ s \ | \ B\in{\cal F}; s\in S\}$ is an additive  
$(G\times\F_{q^n},G\times\{0\},k,1)$-DF. The assertion then follows 
from Lemma \ref{additiveDF}.
\end{proof}

Recall that, for the time being, only three classes of non-trivial additive Steiner 2-designs are known, 
that are classes C1, C2, C3 mentioned in the introduction. The set of their block sizes clearly coincides
with the set $Q \ \cup \ (Q+1)$ where $Q$ is the set of all prime powers.
Thus, for now, we do not have any example of an additive non-trivial Steiner 
2-design whose block size is neither a  prime power nor a prime power plus one.
Also, for $k\in(Q+1)\setminus Q$ we have only one example that is the projective
plane of order $k-1$.
Let us examine which is the very first possible attempt of filling these gaps using the above lemma. 
The first $k$ which is neither a prime power nor a prime power plus one is 15. 
We can try to find a super-regular $2$-$(15q,15,1)$ design using Lemma \ref{additive(kq,k,k,1)},
i.e., via an additive $(\Z_{15}\times\F_q,\Z_{15}\times\{0\},15,1)$-DF 
with $q$ a power of 3 or a power of 5. The first case is ruled out by Theorem \ref{nonexistence} in the next section.
Thus $q$ has to be taken among the powers of 5. 
More precisely, in view of the condition  $q\equiv1$ (mod 14),
we have to take $q=5^{6n}$ for some $n$.  We conclude that 
$2$-$(15\cdot5^6,15,1)$ is the first parameter set of a super-regular Steiner 2-design with 
block size belonging to $(Q+1)\setminus Q$
potentially obtainable  via Lemma \ref{additive(kq,k,k,1)}. Unfortunately, we are
not able to construct a design with these parameters. In the penultimate section 
we indicate a possible attempt to get it by means of a computer search. 
In that same section we will prove the existence of a super-regular $2$-$(15\cdot5^{30},15,1)$ design.

\section{One more necessary condition and the hard case $k=2^n3$ with $n\geq2$}\label{2^n3}
 
The following result will lead to one more condition on the parameters
of a super-regular Steiner 2-design. This result will also imply that a non-trivial
super-regular 2-$(v,k,1)$ design with $k=2^n3$ may be generated by
an additive $(v,k,k,1)$-DF only if a very strong condition on $n$ holds. As a matter
of fact we suspect that this conditions is never satisfied. 
This is why in our main result we are not able to say 
anything about the case $k=2^n3$ which appears to us very hard.

\begin{thm}\label{nonexistence}
If a super-regular $2$-$(v,k,1)$ design is generated by an additive $(v,k,k,1)$-DF
and $k\equiv\pm3$ $($mod $9)$, then ${v\over k}\equiv1$ $($mod $3)$.
\end{thm}
\begin{proof}
Let $\cal F$ be an additive $(G,H,k,1)$-DF generating a $G$-super-regular $2$-$(v,k,1)$ design with $k\equiv\pm3$ $($mod $9)$.
Thus $G$ is a group of order $v\equiv k$ (mod $k(k-1))$, say $v=kv_1$, 
and $H$ is a subgroup of $G$ of order $k=3k_1$ with $k_1$ not divisible by 3.
For what we observed immediately after Proposition \ref{k=4n+2}
the Sylow 3-subgroup of $G$ is elementary abelian. For this reason,
for every two subgroups of $G$ of order 3 there exists an automorphism of $G$
mapping one into the other.
Then, up to isomorphism, we may assume that $G=\Z_3\times G_1$
with $G_1$ of order $k_1v_1$ and $H=\Z_3\times H_1$ with $H_1$ a subgroup of $G_1$ of order $k_1$.

For each $B\in{\cal F}$, let $\overline{B}$ be the $k$-multiset on $\Z_3$
that is the projection of $B$ on $\Z_3$ and set $\overline{\cal F}=\{\overline{B} \ | \ B\in{\cal F}\}$. 
It is clear that $\Delta\overline{\cal F}$ is the projection of $\Delta{\cal F}$ on $\Z_3$.
Thus, considering that $\Delta{\cal F}=(\Z_3\times G_1)\setminus(\Z_3\times H_1)$
by assumption, it is clear that $\Delta\overline{\cal F}$ is $\lambda$ times $\Z_3$ with $\lambda$
equal to the size of $G_1\setminus H_1$, i.e., $\lambda=k_1(v_1-1)$. 
Using some terminology that we will recall in the next section,
$\overline{\cal F}$ is essentially a $(\Z_3,k,\lambda)$ {\it strong difference family}.

Take any block $\overline{B}$ of $\overline{\cal F}$ and for $i=0,1,2$, let $\mu_i$ 
be the multiplicity of $i$ in $\overline{B}$. 
Clearly, we have $\mu_0+\mu_1+\mu_2=k$, hence $\mu_0+\mu_1+\mu_2\equiv0$ (mod 3).
Considering that $\cal F$ is additive, $B$ is zero-sum and then $\overline{B}$ is zero-sum as well.
It follows that $\mu_1+2\mu_2=0$ in $\Z_3$, i.e., $\mu_1\equiv\mu_2$ (mod 3). We easily conclude that
\begin{equation}\label{mu}
\mu_0\equiv\mu_1\equiv\mu_2 \quad {\rm (mod \ 3)}
\end{equation}

Now, let $\nu$ be the multiplicity of zero in $\Delta\overline{B}$ and note that we have 
$$\nu=\mu_0(\mu_0-1)+\mu_1(\mu_1-1)+\mu_2(\mu_2-1),$$
hence $\nu\equiv0$ (mod 3) in view of (\ref{mu}).

Note that $\lambda$ can be seen as the sum of the multiplicities of zero in the lists of
differences of the blocks of $\overline{\cal F}$.
For what we just saw, all these multiplicities are zero (mod 3) and then 
$\lambda\equiv0$ (mod 3). Recalling that $\lambda=k_1(v_1-1)$ we conclude that
$v_1\equiv1$ (mod 3) which is the assertion.
\end{proof}

As a consequence, we get the following non-existence result.

\begin{thm}
A super-regular $2$-$(v,k,1)$ design with $k\equiv\pm3$ $($mod $6)$ and 
${v\over k}\equiv2$ $($mod $3)$ cannot exist.
\end{thm}
\begin{proof}
Assume that there exists a $G$-super-regular $2$-$(v,k,1)$ design $\cal D$
with $v$ and $k$ as in the statement. Then $\cal D$
cannot be generated by an additive $(v,k,k,1)$-DF by Theorem \ref{nonexistence}. 
It follows that $\cal D$ is generated by an additive $(v,\{k^s\},k,1)$-DF for a suitable $s>1$
by Theorem \ref{additiveDF}. On the other hand the hypothesis obviously imply that $v$, as $k$, is divisible 
by 3 but not by 9. Thus $G$ necessarily has only one subgroup of order 3, hence it cannot have a partial spread 
with two distinct members of order $k$. We got a contradiction.
\end{proof}

Each of the following pairs $(v,k)$ satisfies the admissibility conditions $v\equiv k$ (mod $k(k-1))$
and $rad(v)=rad(k)$ given by Proposition \ref{k=4n+2}. Yet, for each of them no super-regular 2-$(v,k,1)$ design
exists in view of the above theorem.
\begin{center}
\begin{tabular}{|c|c|}
\hline
$v$  & $k$\\
\hline
$3\cdot2^6\cdot5^{10}$ & $3\cdot2^2\cdot5$ \\
\hline
$3\cdot2^{18}\cdot11^{10}$ & $3\cdot2^2\cdot11$\\
\hline
$3\cdot5\cdot11^7$ & $3\cdot5\cdot11$\\
\hline
$3\cdot2^{21}\cdot7^{3}$ & $3\cdot2^3\cdot7$\\ 
\hline
$3\cdot5^{22}\cdot13^{4}$ & $3\cdot5\cdot13$\\ 
\hline
$3\cdot2^{26}\cdot5^{6}$ & $3\cdot2^4\cdot5$\\ 
 \hline
\end{tabular}
\end{center}

Another consequence of Theorem \ref{nonexistence} is the following.
\begin{thm}
Let $k=2^n3$ and assume that there exists a super-regular $2$-$(v,k,1)$ design generated by an additive 
$(v,k,k,1)$-DF. Then $v=2^{oi+n}3$ where $o$ is the order of $2$
in the group of units $($mod $k-1)$ and $0\leq i\leq \lfloor{n^2-n\over o}\rfloor$.
\end{thm}
\begin{proof}
Let $\cal D$ be a $G$-super-regular $2$-$(v,k,1)$ design with $k=2^n3$ and 
assume that $\cal D$ is generated by an additive $(G,H,k,1)$-DF so that
$G$ has order $v$ and $H$ is a subgroup of $G$ of order $k$. 
By Proposition \ref{k=4n+2} (ii) and (iii) we have $v=2^a3^b\equiv k$ (mod $k(k-1))$. 
Thus, reducing mod $k$ and mod $k-1$ we respectively get 
\begin{equation}\label{mod k & mod k-1}
2^a3^b\equiv 2^n3 \ (mod \ 2^n3) \quad {\rm and}\quad 2^a3^b\equiv 1 \ (mod \ 2^n3-1)
\end{equation}
From the first of the above congruences we deduce that $a\geq n$ and $b\geq 1$.
By Theorem \ref{nonexistence}
we must have $2^{a-n}3^{b-1}\equiv1$ (mod 3) which implies $b=1$.
Hence $v=2^a3$ with $a\geq n$. 
Multiplying the second congruence in (\ref{mod k & mod k-1}) by $2^n$ (which is the inverse of $3$ mod $k-1$), we get
$2^a\equiv2^n$ (mod $k-1$), i.e., $2^{a-n}\equiv1$ (mod $k-1$). This, by definition of $o$, means that
$a=oi+n$ for some integer $i$. Hence we have
\begin{equation}\label{v1}
v=2^{oi+n}3
\end{equation}
Now let $2^t$ be the order of $I(G)$ and recall that $I(G)$ is necessarily contained in $H$ so that we have $t\leq n$.
Up to isomorphism, by the fundamental theorem on abelian groups, we have 
$G=\Z_{2^{\alpha_1}}\times\dots\times\Z_{2^{\alpha_t}}\times\Z_3$ for a suitable
$t$-tuple $(\alpha_1,\dots,\alpha_t)$ of positive integers summing up to $a$. 
For $i=1,\dots,t$, there are elements of $G$ of order $2^{\alpha_i}$; for instance the element
whose $i$th coordinate is 1 and all the other coordinates are zero. Hence
$2^{\alpha_i}$ divides $k$ by Proposition \ref{k=4n+2}(i) and then $\alpha_i\leq n$ for $i=1,\dots,t$.
We deduce that we have 
\begin{equation}\label{v2}
v=|G|\leq(2^n)^t3\leq 2^{n^2}3
\end{equation}
Comparing (\ref{v1}) and (\ref{v2}) we get $oi+n\leq n^2$, i.e., $i\leq \lfloor{n^2-n\over o}\rfloor$
and the assertion follows.
\end{proof}

\begin{cor}
If $k=2^n3$ and there exists a non-trivial super-regular $2$-$(v,k,1)$ design 
generated by an additive $(v,k,k,1)$-DF, then the order of $2$ 
in the group of units of $\Z_{k-1}$ is less than $n^2-n$.
\end{cor}

We suspect that the order of 2 in the group of units of
$\Z_{2^n3-1}$ is always greater than $n^2-n$ but we are not able to prove it. For now, we are able to say that it is
true for $n\leq1000$ (checked by computer) and whenever $2^n3-1$ has a prime factor greater than $(n^2-n)^2$;
this is a consequence of a result proved in \cite{M}
according to which the order of 2 modulo an odd prime $p$ is almost always as large as the square root of $p$. 
Thus, for now, we can state the following.

\begin{rem}
Let $k=2^n3$ with $n\leq1000$ or $k$ has a prime factor greater than $(n^2-n)^2$. 
Then there is no value of $v$ for which a putative 
non-trivial super-regular $2$-$(v,k,1)$ design
may be generated by an additive $(v,k,k,1)$-DF.
\end{rem}

The above leads us to believe that the existence of a non-trivial super-regular $2$-$(v,2^n3,1)$ design
generated by a $(v,k,k,1)$-DF is highly unlikely. 
On the other hand such a design might be obtained via a difference 
family relative to a partial spread of size greater than 1.
For instance, we cannot rule out that there exists a $G$-super-regular $2$-$(3^{9}4^4,12,1)$ design
generated by an additive $(G,{\cal H},12,1)$-DF with $G=\F_{3^{9}}\times\F_{4^4}$ and $\cal H$
a partial spread of $G$ of type $\{12^{85}\}$.
Indeed $G$ satisfies conditions (ii), (iii) of Propositions \ref{k=4n+2} and the necessary
conditions (\ref{necessaryspread}) are also satisfied with an $\cal H$ constructible
as follows.
Take a (full) spread ${\cal H}_1$ of $\F_{4^4}$ consisting of
subgroups of $\F_{4^4}$ of order 4 and note that it has size ${4^4-1\over3}=85$.
Now take the (full) spread ${\cal H}_2$ of $\F_{3^{9}}$ consisting of all
subgroups of $\F_{3^{9}}$ of order 3 which has size ${3^{9}-1\over2}>85$. 
Thus it is possible to choose an injective map $f:{\cal H}_1\longrightarrow{\cal H}_2$ and
we can take ${\cal H}:=\{H\times f(H) \ | \ H\in{\cal H}_1\}$.
On the other hand to realize an additive $(G,{\cal H},12,1)$-DF with $G$ and ${\cal H}$
as above appears to be unfeasible; suffice it to say that it would have 38,166 base blocks. 
Also, the fact that the literature is completely lacking of constructions for $(v,\{k^s\},k,1)$
difference families with $s>1$, further underlines the difficulty of the problem.

\section{Strong difference families}\label{sectionSDF}

In view of Lemma \ref{additive(kq,k,k,1)} our target will be the construction of additive  
$(G\times\F_q,G\times\{0\},k,1)$ difference families with $G$ of order $k$ and $q$ a power of a prime divisor of $k$.
For this, we need one more variant of a difference family, that is a {\it strong difference family}.

The notion of list of differences of a subset of a group $G$ can be naturally
generalized to that of list of differences of a multiset on $G$ as follows.
If $B=\{b_1,\dots,b_k\}$ is a multiset on a group $G$, then the list of differences
of $B$ is the multiset $\Delta B$ of all possible differences $b_i-b_j$ with $(i,j)$
an ordered pair of distinct elements of $\{1,\dots,k\}$.

It is evident that the multiplicity of zero in $\Delta B$ is even. Indeed if $b_i-b_j=0$,
then $b_j-b_i=0$ as well. It is also evident that this multiplicity is equal to zero
if and only if $B$ does not have repeated elements, i.e., $B$ is a set.

By list of differences of a collection ${\cal F}$ of multisets on $G$ one
means the multiset union $\Delta{\cal F}=\biguplus_{B\in {\cal F}}\Delta B$.

\begin{defn}\label{SDF}
Let $G$ be a group of order $v$ and let ${\cal F}$ be a collection of
$k$-multisets on $G$. One says that ${\cal F}$ is a $(v,k,\lambda)$ strong difference family
in $G$ (or briefly a $(G,k,\lambda)$-SDF) if $\Delta{\cal F}$ covers 
every element of $G$ ($0$ included) exactly $\lambda$ times.
\end{defn}
Note that if $s$ is the number of blocks of a $(G,k,\lambda)$-SDF, then we necessarily
have $\lambda|G|=sk(k-1)$.

A SDF with only one block is called a {\it difference multiset} \cite{B99} or also
a {\it difference cover} \cite{Arasu}. 

\begin{ex}\label{5,5,4}
Take the $5$-multiset $B=\{0,1,1,4,4\}$ on $\Z_5$. 
Looking at its ``difference table"
\begin{center}
\begin{tabular}{|c|||c|c|c|c|c|c|c|c|c|c|c|c|c|c}
\hline {$$} & {\scriptsize$0$} & {\scriptsize$1$} & {\scriptsize$1$} & {\scriptsize$4$} & {\scriptsize$4$}   \\
\hline\hline\hline {\scriptsize$0$} & $\bullet$ & $\bf4$ & $\bf4$ & $\bf1$ & $\bf1$  \\
\hline {\scriptsize$1$} & $\bf1$ & $\bullet$ & $\bf0$ & $\bf2$ & $\bf2$  \\
\hline {\scriptsize$1$} & $\bf1$ & $\bf0$ & $\bullet$ & $\bf2$ & $\bf2$  \\
\hline {\scriptsize$4$} & $\bf4$ & $\bf3$ & $\bf3$ & $\bullet$ & $\bf0$  \\
\hline {\scriptsize$4$} & $\bf4$ & $\bf3$ & $\bf3$ & $\bf0$ & $\bullet$  \\
\hline
\end{tabular}\quad\quad\quad
\end{center}
we see that the singleton $\{B\}$ is a $(5,5,4)$-SDF in $\Z_5$. 
\end{ex}

Throughout the paper, the union of $n$ copies of a set or multiset $S$ will be denoted by
$\underline{n}S$. Thus the difference multiset of the previous example can be denoted
as $\{0\} \ \cup \ \underline{2}\{1,4\}$. Much more in general, we recall 
that if $q$ is an odd prime power and $\F_q^\Box$ is the set of non-zero squares of $\F_q$, 
then $\{0\} \ \cup \ \underline{2}\F_q^\Box$ is the so-called {\it $(q,q,q-1)$
Paley difference multiset of the first type} \cite{B99}.

We will say that a multiset on a group $G$ is zero-sum if the sum of all its elements
(counting their multiplicities) is zero. 
A SDF in $G$ will be said {\it additive} if
all its members are zero-sum. In view of Fact \ref{fact1} the Paley $(q,q,q-1)$ difference multisets
of the first type are additive provided that $q\neq3$.

\medskip
Strong difference families are a very useful tool to construct relative difference families.
Even though they were implicitly considered in some older literature, they have been formally introduced 
for the first time by the first author in \cite{B99}. After that, they turned out to be crucial in many constructions in design theory
(see, e.g., \cite{BBGRT,BuGio,BP,BYW,CCFW,CFW1,CFW2,FW,Momihara,YYL}).

The following construction explains how to use strong difference families in order to construct
relative difference families.

\begin{constr}\label{constr}
Let $\Sigma=\{B_1,\dots,B_s\}$ be a $(G,k,\lambda)$-SDF and let $q\equiv1$ (mod $\lambda$) be a prime
power.
Lift each block $B_h=\{b_{h1},\dots,b_{hk}\}$ of $\Sigma$ to a subset $\ell(B_h)=\{(b_{h1},\ell_{h1}),\dots,(b_{hk},\ell_{hk})\}$
of $G\times\F_q$. By definition of a strong difference family,
we have $\Delta{\cal F}=\biguplus_{g\in G}\{g\}\times\Delta_g$ where each $\Delta_g$ is a $\lambda$-multiset on $\F_q$. 
Hence, if the liftings have been done appropriately, it may happen that there exists a ${q-1\over\lambda}$-subset $M$ 
of $\F_q^*$ such that $\Delta_g\cdot M=\F_q^*$ for each $g\in G$.
In this case, it is easy to see that $${\cal F}=\bigl{\{}\{(b_{h1},\ell_{h1}m),\dots,(b_{hk},\ell_{hk}m)\} \ | \ 1\leq h\leq s; m\in M\bigl{\}}$$
is a $(G\times\F_q,G\times\{0\},k,1)$-DF. This DF is clearly additive in the additional hypothesis that $\Sigma$ is additive 
and each $\ell(B_h)$ is zero-sum. 
\end{constr}

In most of the cases the above construction is applied when each $\Delta_g$ is a complete system of representatives
for the cosets of the subgroup $C^\lambda$ of $\F_q^*$ of index $\lambda$, that is the group of non-zero $\lambda$-th
powers of $\F_q$. Indeed in this case we have $\Delta_g\cdot M=\F_q^*$ for each $g\in G$
with $M=C^\lambda$. Note, however, that $\Delta_g$ is of the form 
$\{1,-1\}\cdot\overline{\Delta}_g$ for every $g\in I(G)$, hence it contains pairs $\{x,-x\}$ of opposite elements. Thus, if the elements of $\Delta_g$ belong
to pairwise distinct cosets of $C^\lambda$, we necessarily have $-1\notin C^\lambda$, i.e., 
$q\equiv \lambda+1$ (mod $2\lambda)$. This explains why in the next Theorems \ref{SDF->DF} and \ref{additive version} we require
that this congruence holds.

\section{Anomalous $2$-$(q^n,q,1)$ designs}

Let us say that a $2$-$(q^n,q,1)$ design is {\it anomalous} if it is $\F_{q^n}$-super-regular but not isomorphic
to the design of points and lines of AG$(n,q)$.

\begin{prop}
If there exists an anomalous $2$-$(q^n,q,1)$ design, then there exists
an anomalous $2$-$(q^m,q,1)$ design for any $m\geq n$.
\end{prop}
\begin{proof}
Let $V$ be the $n$-dimensional subspace of AG$(m,q)$ defined
by the equations $x_i=0$ for $n+1\leq i\leq m$. Take the {\it standard} $2$-$(q^m,q,1)$ design $(\F_q^m,{\cal B})$ 
and replace all its blocks contained in $V$ with the blocks of
an anomalous $2$-$(q^n,q,1)$ design. We get, in this way, the block-set
of an anomalous $2$-$(q^m,q,1)$ design.
\end{proof}

In the next theorem we put into practice Lemma \ref{additive(kq,k,k,1)} and Construction \ref{constr} to get 
an anomalous $2$-$(p^3,p,1)$ design for $p=5$ and $p=7$. Our proof is a slight modification of the 
construction for regular $2$-$(pq,p,1)$ designs in \cite{B&B} (improved in \cite{cyclotomic}) 
with $p$ and $q$ prime powers, $q\equiv1$ (mod $p-1$). In our construction below $q$ coincides with $p^2$.
\begin{thm}\label{125,5,1}
There exists an anomalous $2$-$(p^3,p,1)$ design for $p=5$ and $p=7$.
\end{thm}
\begin{proof}
By Lemma \ref{additive(kq,k,k,1)} a super-regular $2$-$(5^3,5,1)$ design can be realized by
means of an additive $(\Z_5\times\F_{25},\Z_5\times\{0\},5,1)$-DF.
We can obtain several DFs of the required kind using Construction \ref{constr} with
$\Sigma$ the additive $(5,5,4)$ difference multiset $B=\{0,1,1,4,4\}$ of Example \ref{5,5,4}.
For instance, let us lift $B$ to the subset $\ell(B)$ of $\Z_5\times\F_{25}$
$$\ell(B)=\{(0,0),(1,1),(1,-1),(4,\ell),(4,-\ell)\}$$
with $\ell$ a root of the primitive polynomial $x^2+x+2$.
It is readily seen that $\ell(B)$ is zero-sum.
Looking at its difference table
\begin{center}
\begin{tabular}{|c|||c|c|c|c|c|c|c|c|c|c|c|c|c|c}
\hline {$$} & {\scriptsize$(0,0)$} & {\scriptsize$(1,1)$} & {\scriptsize$(1,-1)$} & {\scriptsize$(4,\ell)$} & {\scriptsize$(4,-\ell)$}   \\
\hline\hline\hline {\scriptsize$(0,0)$} & $\bullet$ & $\bf(4,-1)$ & $\bf(4,1)$ & $\bf(1,-\ell)$ & $\bf(1,\ell)$  \\
\hline {\scriptsize$(1,1)$} & $\bf(1,1)$ & $\bullet$ & $\bf(0,2)$ & $\bf(2,1-\ell)$ & $\bf(2,1+\ell)$  \\
\hline {\scriptsize$(1,-1)$} & $\bf(1,-1)$ & $\bf(0,-2)$ & $\bullet$ & $\bf(2,-1-\ell)$ & $\bf(2,-1+\ell)$  \\
\hline {\scriptsize$(4,\ell)$} & $\bf(4,\ell)$ & $\bf(3,\ell-1)$ & $\bf(3,1+\ell)$ & $\bullet$ & $\bf(0,2\ell)$  \\
\hline {\scriptsize$(4,-\ell)$} & $\bf(4,-\ell)$ & $\bf(3,-1-\ell)$ & $\bf(3,-\ell+1)$ & $\bf(0,-2\ell)$ & $\bullet$  \\
\hline
\end{tabular}\quad\quad\quad
\end{center}
we see that $\displaystyle\Delta\ell(B)=\bigcup_{g=0}^4 \{g\}\times\Delta_g$ with

\smallskip
$\Delta_0=\{1,-1\}\cdot\{2,2\ell\};$

\smallskip
$\Delta_1=\Delta_4=\{1,-1\}\cdot\{1,\ell\};$

\smallskip
$\Delta_2=\Delta_3=\{1,-1\}\cdot\{\ell-1,\ell+1\}.$

Now note that each of the $2$-sets $\overline{\Delta}_0=\{2,2\ell\}$, $\overline{\Delta}_1=\{1,\ell\}$ 
and $\overline{\Delta}_2=\{\ell-1,\ell+1\}$ contains a non-zero square and a non-square of $\F_{25}$. Thus, if $M$ is a 
complete system of representatives for the cosets of $\{1,-1\}$ in $\F_{25}^\Box$,
we clearly have $\Delta_g\cdot M=\F_{25}^*$. Hence
$${\cal F}=\bigl{\{}\{(0,0),(1,m),(1,-m),(4,\ell m),(4,-\ell m)\} \ | \ m\in M\bigl{\}}$$
is an additive $(\Z_5\times\F_{25},\Z_5\times\{0\},5,1)$-DF. If we take, for instance, $M=\{\ell^{2i} \ | \ 0\leq i\leq 5\}$
then the blocks of $\cal F$, written in additive notation, are the following:
$$B_1=\{(0,0,0),(1,0,1),(1,0,4),(4,1,0),(4,4,0)\};$$
$$B_2=\{(0,0,0),(1,4,3),(1,1,2),(4,4,2),(4,1,3)\};$$
$$B_3=\{(0,0,0),(1,3,2),(1,2,3),(4,4,4),(4,1,1)\};$$
$$B_4=\{(0,0,0),(1,0,2),(1,0,3),(4,2,0),(4,3,0)\};$$ 
$$B_5=\{(0,0,0),(1,3,1),(1,2,4),(4,3,4),(4,2,1)\};$$
$$B_6=\{(0,0,0),(1,1,4),(1,4,1),(4,3,3),(4,2,2)\}.$$

We can check, by hand, that the super-regular $2$-$(125,5,1)$ design $\cal D$ generated by $\cal F$ is anomalous. 
Assume for contradiction that it is isomorphic to the point-line design of AG$(3,5)$. It is then
natural to speak of {\it lines} of $\cal D$ rather than blocks. Also, it makes sense to speak 
of the {\it planes} of $\cal D$ and a line containing two distinct points of a plane $\pi$
is clearly contained in $\pi$.

Let $\pi$ be the plane of $\cal D$ containing the two lines through the origin 
$B_0=\{(0,0,0),(1,0,0),(2,0,0),(3,0,0),(4,0,0)\}$ and $B_1$. Of course, if ${\cal B}_\pi$ is the set of lines
of $\cal D$ contained in $\pi$, then $(\pi,{\cal B}_\pi)$ is isomorphic to the affine plane over $\F_5$.
The line through $(1,0,0)\in B_0$ and $(1,0,1)\in B_1$ is $$C=B_4+(0,0,3)=\{(0,0,3),{\bf(1,0,0)},{\bf(1,0,1)},(4,2,3),(4,3,3)\}$$
and belongs to ${\cal B}_\pi$ since it joins two points of $\pi$.
The line through $(1,0,4)\in B_1$ and $(0,0,3)\in C$ is $$D=B_1+(0,0,3)=\{{\bf(0,0,3)},{\bf(1,0,4)},(1,0,2),(4,1,3),(4,4,3)\}.$$
The line through $(1,0,4)\in B_1$ and $(4,2,3)\in C$ is $$D'=B_6+(0,4,0)=\{(0,4,0),{\bf(1,0,4)},(1,3,1),{\bf(4,2,3)},(4,1,2)\}.$$
These two lines $D$ and $D'$ also belong to ${\cal B}_\pi$ since they also join two points of $\pi$. We also note that they are both 
disjoint with the line $B_0 \in {\cal B}_\pi$. This contradicts the Euclid's parallel axiom: 
there is a point of $\pi$ (that is $(1,0,4)$) and two distinct lines of $\pi$ through this point 
($D$ and $D'$) which are both disjoint with a line of $\pi$ (that is $B_0$).

\medskip
Now consider the $(7,7,6)$ Paley difference multiset of the first type, that is
$\{0\} \cup \underline{2}\{1,2,4\}$, and apply Construction \ref{constr} lifting it to a suitable 7-subset of $\F_{49}$.
Without entering all the details, we just list the base blocks of the resultant
$(\Z_7^3,\Z_7\times\{0\}\times\{0\},7,1)$-DF.
$$\{(0,0,0),(1,1,0),(1,6,0),(2,2,1),(2,5,6),(4,2,0),(4,5,0)\}$$
$$\{(0,0,0),(1,2,4),(1,5,3),(2,0,3),(2,0,4),(4,4,1),(4,3,6)\}$$
$$\{(0,0,0),(1,2,2),(1,5,5),(2,2,6),(2,5,1),(4,4,4),(4,3,3)\}$$
$$\{(0,0,0),(1,3,5),(1,4,2),(2,1,6),(2,6,1),(4,6,3),(4,1,4)\}$$
$$\{(0,0,0),(1,0,1),(1,0,6),(2,6,2),(2,1,5),(4,0,2),(4,0,5)\}$$
$$\{(0,0,0),(1,3,2),(1,4,5),(2,4,0),(2,3,0),(4,6,4),(4,1,3)\}$$
$$\{(0,0,0),(1,5,2),(1,2,5),(2,1,2),(2,6,5),(4,3,4),(4,4,3)\}$$
$$\{(0,0,0),(1,2,3),(1,5,4),(2,1,1),(2,6,6),(4,4,6),(4,3,1)\}$$ 
One can check that the design generated by the above DF 
is anomalous with the same isomorphism test used for 
getting the anomalous $2$-$(5^3,5,1)$ design.
\end{proof} 

The above results allow us to state the following.

\begin{cor}
There exists an anomalous $2$-$(p^n,p,1)$ design for $p\in\{5,7\}$ and any integer $n\ge3$.
\end{cor}

We tried to get an anomalous $2$-$(11^3,11,1)$ design with the same method used in the proof of Theorem \ref{125,5,1},
i.e., by means of a suitable lifting of the $(11,11,10)$ Paley difference multiset $\{0\}\cup\{1,3,4,5,9\}$, but we fail.

\section{Cyclotomy}

Starting from the fundamental paper of Wilson \cite{W}, cyclotomy has been very often
crucial in the construction of many classes of difference families. Here it is also crucial
for getting a good lifting of a SDF as required by Construction \ref{constr}. 

Given a prime power $q\equiv1$ (mod $\lambda$), let $C^\lambda$ be the subgroup of $\F_q^*$ of index $\lambda$.
If $r$ is a fixed primitive element of $\F_q$, then $\{r^iC^\lambda \ | \ 0\leq i\leq \lambda-1\}$ is the 
set of cosets of $C^\lambda$ in $\F_q^*$.  
For $i=0,1,\dots,\lambda-1$, the coset $r^iC^\lambda$ will be denoted by $C^\lambda_i$ and it is 
called {\it the $i$-th cyclotomic class of order $\lambda$}. Note that we have
$C^\lambda_i\cdot C^\lambda_j=C^\lambda_{i+j \ (mod \lambda)}$.
We will need the following lemma deriving from the theorem of Weil on multiplicative character sums
(see \cite{LN}, Theorem 5.41).
\begin{lem}\label{BP} {\rm\cite{BP}}
Let $q\equiv 1 \pmod{\lambda}$ be a prime power and let $t$ be a positive integer.
Then, for any $t$-subset $C=\{c_1,\dots,c_t\}$ of $\F_q$ and for any ordered $t$-tuple $(\gamma_1,\dots,\gamma_t)$ of $\Z_\lambda^t$,
the set $X:=\{x\in \F_q: x-c_i\in C_{\gamma_i}^\lambda \,\,{\rm  for }\,\, i=1,\dots,t \}$
has arbitrarily large size provided that $q$ is sufficiently large.

In particular, we have $|X|>2\lambda^{t-1}$ for $q>t^2\lambda^{2t}$. 
\end{lem}

In most cases the above lemma has been used to prove that the set $X$ is not empty. 
But this is not enough for our purposes.
The last sentence in the above statement is formula (2) in \cite{BP}.

The following theorem is essentially Corollary 5.3 in \cite{BP} where it appeared as a special consequence
of a more general result. Here, for convenience of the reader, it is better to show its proof directly.
Then we will see how this proof can be modified in order to get its {\it additive version}.
\begin{thm}\label{SDF->DF}
If there exists a $(G,k,\lambda)$-SDF, then there exists a $(G\times\F_q,G\times\{0\},k,1)$-DF
for every prime power $q\equiv\lambda+1$ $($mod $2\lambda)$ provided that $q>(k-1)^2\lambda^{2k-2}$.
\end{thm}
\begin{proof}
Let $\Sigma=\{B_1,\dots,B_s\}$ be a $(G,k,\lambda)$-SDF with $B_h=\{b_{h1},\dots,b_{hk}\}$
for $1\leq h\leq s$.
Let $T$ be the set of all triples $(h,i,j)$ with $h\in\{1,\dots,s\}$ and $i$, $j$
distinct elements of $\{1,\dots,k\}$.
For every $g\in G$, let $T_g$ be the set of triples $(h,i,j)$ of $T$ such that $b_{h,i}-b_{h,j}=g$.
Note that $\bigcup_{g\in G}T_g$ is a partition of $T$ and that each $T_g$ has size $\lambda$
by definition of a $(G,k,\lambda)$-SDF.
Thus it is possible to choose a map $\psi: T \longrightarrow \Z_\lambda$
satisfying the following conditions: 

1) the restriction $\psi |_{T_g}$ is bijective for any $g\in G$;

2) $\psi(h,j,i)=\psi(h,i,j)+\lambda/2$ for every pair of distinct $i$, $j$.

As a matter of fact the number $\Psi$ of all maps $\psi$ satisfying the above conditions is huge.
If $\lambda=2\mu$ and $|G|=2^nm$ where $2^n$ is the order of $I(G)$,
it is easy to see that $\Psi=\lambda!^{2^{n-1}(m-1)}(2^\mu\mu!)^{2^n}$.

Now lift each $B_h$ to a subset $\ell(B_h)=\{(b_{h1},\ell_{h1}),\dots,(b_{hk},\ell_{hk})\}$ of $G\times \F_q$
by taking the first element $\ell_{h,1}$ arbitrarily and then by taking the other elements 
$\ell_{h,2}$, $\ell_{h,3}$, \dots, $\ell_{h,k}$ iteratively, one by one, according to the rule that
once that $\ell_{h,i-1}$ has been chosen, we pick $\ell_{h,i}$ arbitrarily in the set 
$$X_{h,i}=\{x\in \F_q \ : \ x-\ell_{h,j}\in C^\lambda_{\psi(h,i,j)} \quad {\rm for} \ 1\leq j\leq i-1\}.$$
Note that $\{\ell_{h,1},...,\ell_{h,i-1}\}$ is actually a set, i.e., it does not have repeated elements.
Indeed given two elements $j_1<j_2$ in $\{1,\dots,i-1\}$, we have 
$\ell_{h,j_2}-\ell_{h,j_1} \in C^\lambda_{\psi(h,j_2,j_1)}$ since $\ell_{h,j_2}$
has been picked in $X_{h,j_2}$. Thus we cannot have $\ell_{h,j_2}=\ell_{h,j_1}$.
It follows that $X_{h,i}$ is not empty by Lemma \ref{BP}, hence an element $\ell_{h,i}$
with the above requirement can be actually chosen. 

Also note that we have 
\begin{equation}\label{ell-ell}
\ell_{h_,i}-\ell_{h,j}\in C^\lambda_{\psi(h,i,j)}\quad \forall (h,i,j)\in T
\end{equation}
This is clear if $i>j$ considering the rule that we followed for selecting the $\ell_{h,i}$'s.
If $i<j$, for the same reason, we have $\ell_{h_,j}-\ell_{h,i}\in C^\lambda_{\psi(h,j,i)}$, i.e.,
$\ell_{h_,j}-\ell_{h,i}\in C^\lambda_{\psi(h,i,j)+\lambda/2}$ in view of the second property of $\psi$.
Multiplying by $-1$ and considering that $-1\in C^\lambda_{\lambda/2}$ since 
$q\equiv\lambda+1$ (mod $2\lambda$), we get (\ref{ell-ell}) again.

We finally note that we have $\biguplus_{h=1}^s\Delta\ell(B_h)=\biguplus_{g\in G}\{g\}\times\Delta_g$
with $\Delta_g=\{\ell_{h,i}-\ell_{h,j} \ | \ (h,i,j)\in T_g\}$. Thus, in view of (\ref{ell-ell}) and 
the first property of $\psi$, we see that  $\Delta_g$ is a complete system of representatives 
for the cyclotomic classes of order $\lambda$ whichever is $g\in G$. At this point we get the 
required  $(G\times\F_q,G\times\{0\},k,1)$-DF by applying Construction \ref{constr} as pointed 
out at the end of Section \ref{sectionSDF}.
\end{proof}

The additive version of the above theorem is straightforward in the case that $rad(q)$ is not a divisor
of $k$. On the contrary, if $rad(q)$ divides $k$, which in view of Lemma \ref{additive(kq,k,k,1)} is the case we are interesting in, we  
have to lift the base blocks of the given additive SDF much more carefully. Also, we need to raise the bound on $q$
significantly, and to ensure that the order of $G$ is not too large.

\begin{thm}\label{additive version}
Assume that there exists an additive $(G,k,\lambda)$-SDF of size $s$ with $k\neq3$ and
let $q\equiv\lambda+1$ (mod $2\lambda)$ be a prime power. Then there exists 
an additive $(G\times\F_q,G\times\{0\},k,1)$-DF in each of the following cases:
\begin{itemize}
\item[(i)] $rad(q)$ does not divide $k$ and $q>(k-1)^2\lambda^{2k-2}$;
\item[(ii)] $rad(q)$ divides $k$, $|G|<2\lambda^{2k-5}s$ and $q>(2k-3)^2\lambda^{4k-6}$.
\end{itemize}
\end{thm}
\begin{proof} Let  $\Sigma=\{B_1,\dots,B_s\}$ be a $(G,k,\lambda)$-SDF as in the proof of the previous theorem
and let $q\equiv\lambda+1$ (mod $2\lambda$) be a prime power.

\smallskip
(i) $k$ is not divisible by $rad(q)$, and $q>(k-1)^2\lambda^{2k-2}$.\\
Take a $(G\times\F_q,G\times\{0\},k,1)$-DF, say $\cal F$, which exists by Theorem \ref{SDF->DF}.
For every block $B\in{\cal F}$, let $\sigma_B$ be the sum of the second coordinates
of all elements of $B$ and set $B'=B+(0,-{\sigma_B\over k})$. It is evident that
$\{B' \ | \ B \in{\cal F}\}$ is an additive $(G\times\F_q,G\times\{0\},k,1)$-DF.

\smallskip
(ii) $rad(q)$ divides $k$, $|G|<2\lambda^{2k-5}s$, and $q>(2k-3)^2\lambda^{4k-6}$.\\
We keep the same notation as in the proof of the above theorem and the procedure for
getting $\ell(B_h)$ will be exactly the same until determining the element $\ell_{h,k-4}$.
After that we have to be much more careful in picking the last four elements
$\ell_{h,k-3}$, $\ell_{h,k-2}$, $\ell_{h,k-1}$ and $\ell_{h,k}$.
In the following, we set $\sigma_{h,i}=\sum_{j=1}^i\ell_{h,j}$ once that all $\ell_{h,j}$'s with $1\leq j\leq i$
have been chosen.

\eject
Choice of $\ell_{h,k-3}$.

\noindent
If $rad(q)\neq 3$, just proceed as in the proof of Theorem \ref{SDF->DF}; we can take $\ell_{h,k-3}$
in $X_{h,k-3}$ arbitrarily. If $rad(q)=3$ we take it in $X_{h,k-3}\setminus\{-\sigma_{h,k-4}\}$.
Note that $rad(q)=3$ implies $k>4$ since we have $k\neq3$ by assumption, hence 
it makes sense to consider the sum $\sigma_{h,k-4}$.

\smallskip
Choice of $\ell_{h,k-2}$.

\noindent
We pick this element in $X_{h,k-2}\setminus Y$, where $Y$ is the union of the sets
$$Y_1=\{-\sigma_{h,k-3}-\ell_{h,i}-\ell_{h,j} \ | \ 1\leq i\leq j\leq k-3\},$$
$$Y_2=\{- \sigma_{h,k-3} - \ell_{h,i} \ | \ 1\leq i\leq k-3\},\quad\quad Y_3={1\over2}Y_2,$$
and, only in the case that $rad(q)\neq3$, the singleton $Y_4=\{-{\sigma_{h,k-3}\over3}\}$.
Note that this selection can be done since $|X_{h,k-2}|>|Y|$.
Indeed we have $|X_{h,k-2}|>2\lambda^{k-4}$ by Lemma \ref{BP} and $2\lambda^{k-4}>{\lambda|G|\over s}$
in view of the upper bound on the order of $G$. Also, we have ${\lambda|G|\over s}=k(k-1)$ since,
as observed after Definition \ref{SDF}, we have $\lambda|G|=sk(k-1)$. Finally, it is evident that $Y$
has size less than $k(k-1)$.

\smallskip
Choice of $\ell_{h,k-1}$.

\noindent
We pick this element in the set 
$$X'_{h,k-1}=\{x\in \F_q \ : \ x-c_{h,j}\in C^\lambda_{\gamma_{h,j}} \ {\rm for} \ 1\leq j\leq 2k-3\}$$
with the pairs $(c_{h,j},\gamma_{h,j})$ defined as follows:
$$c_{h,j}=\ell_{h,j} \quad{\rm and}\quad \gamma_{h,j}=\psi(h,k-1,j) \quad {\rm for} \ 1\leq j\leq k-2;$$
$$c_{h,k-2+j}=-\sigma_{h,k-2}-\ell_{h,j} \quad{\rm and}\quad \gamma_{h,k-2+j}=\psi(h,k,j)+{\lambda\over2} \quad{\rm for} \ 1\leq j\leq k-2;$$
$$c_{h,2k-3}=-{\sigma_{h,k-2}\over2} \quad{\rm and}\quad \gamma_{h,2k-3}=\psi(h,k,k-1)-\alpha$$
where $C^\lambda_\alpha$ is the cyclotomic class of order $\lambda$ containing $-2$.

Note that the first $k-2$ conditions required for the generic element of $X'_{h,k-1}$
are exactly the conditions for the generic element of $X_{h,k-1}$. Thus $X'_{h,k-1}$
is a subset of $X_{h,k-1}$.

Assume that $c_{h,j_1}=c_{h,j_2}$ with $1\leq j_1< j_2\leq 2k-3$. 

If $j_2\leq k-2$, then we have
$\ell_{h,j_1}=\ell_{h,j_2}$ which contradicts the fact that 
$\ell_{h,j_2}-\ell_{h,j_1} \in C^\lambda_{\psi(h,j_2,j_1)}$ (recall indeed that $\ell_{h,j_2}$ is in $X_{h,j_2}$).

For the same reason, we cannot have $k-1\leq j_1< j_2\leq 2k-4$. 

If $j_1=k-2$ and $j_2=2k-3$ we get $-\sigma_{h,k-3}-3\ell_{h,k-2}$. 
If $rad(q)=3$, this means $\sigma_{h,k-3}=0$, hence $\ell_{h,k-3}=-\sigma_{h,k-4}$ contradicting
the choice of $\ell_{h,k-3}$ in this case. If $rad(q)\neq3$, then we would have $\ell_{h,k-2}=-{\sigma_{h,k-3}\over3}$
contradicting the choice of $\ell_{h,k-2}$ in this case.

In all the remaining cases the reader can check that we would get $\ell_{h,k-2}\in Y$. 
On the other hand, $\ell_{h,k-2}$ had been picked out of $Y$ on purpose. We conclude that the 
$c_{h,j}$'s ($j=1,2,\dots,2k-3)$ are pairwise distinct.
Thus, Lemma \ref{BP} and the assumption $q>(2k-3)^2\lambda^{4k-6}$  guarantee that $X'_{h,k-1}$ is 
 not empty and the selection of $\ell_{h,k-1}$ described above can be actually done.
 
\smallskip
Choice of $\ell_{h,k}$.

\noindent
Take $\ell_{h,k}=-\sigma_{h,k-1}$. This last (obligatory) choice assures that $\ell(B_h)$ is zero-sum; 
the sum of the first coordinates of all its elements is zero because $\Sigma$ is additive, and the sum of the second coordinates 
of all its elements is $\sigma_{h,k}=\sigma_{h,k-1}+\ell_{h,k}=0$.

\smallskip
It is evident that  $\ell_{h,i}\in X_{h,i}$ for $1\leq i\leq k-1$. As a  consequence of the fact 
that $\ell_{h,k-1}\in X'_{h,k-1}$, we show that this is true also for $i=k$, i.e., that 
we have $\ell_{h,k}-\ell_{h,j}\in C^\lambda_{\psi(h,k,j)}$ for $1\leq j\leq k-1$.  

$1\leq j\leq k-2$: by definition of $X'_{h,k-1}$, we have 
\begin{equation}\label{ell_{h,k}}
\ell_{h,k-1}-c_{h,k-2+j}\in C^\lambda_{\psi(h,k,j)+\lambda/2}.
\end{equation} Now note that $\ell_{h,k-1}-c_{h,k-2+j}=-\ell_{h,k}+\ell_{h,j}$
by the definitions of $c_{h,k-2+j}$ and $\ell_{h,k}$. Thus, multiplying (\ref{ell_{h,k}})
by $-1$ and recalling that $-1\in C^\lambda_{\lambda/2}$, we actually get $\ell_{h,k}-\ell_{h,j}\in C^\lambda_{\psi(h,k,j)}$.

$j=k-1$:
considering the last condition required for the generic element of $X'_{h,k-1}$, we have
$\ell_{h,k-1}+{\sigma_{h,k-2}\over2}\in C^\lambda_{\psi(h,k,k-1)-\alpha}$. 
Multiplying by $-2$ and remembering that $-2\in C^\lambda_\alpha$
we get $-2\ell_{h,k-1}-\sigma_{h,k-2}\in C^\lambda_{\psi(h,k,k-1)}$ which is what
we wanted. Indeed, by definition of $\ell_{h,k}$, we have $-2\ell_{h,k-1}-\sigma_{h,k-2}=\ell_{h,k}-\ell_{h,k-1}$.

\smallskip
We conclude that the above constructed liftings are in the same situation of the liftings constructed in 
the proof of Theorem \ref{SDF->DF}, i.e., (\ref{ell-ell}) holds. 
Thus, reasoning as at the end of that proof, we can say that they 
form a $(G\times\F_q,G\times\{0\},k,1)$-DF. The assertion follows considering that each of
these liftings is zero-sum.
\end{proof}

We are going to see that the above theorem allows to obtain a difference family as required in Lemma \ref{additive(kq,k,k,1)}
as soon as one has an additive $(G,k,\lambda)$-SDF with $G$ a zero-sum group of order $k$ and 
$\lambda$ not divisible by $rad(k)$. This will be the crucial ingredient for proving our main result.

\begin{lem}\label{crucial}
Assume that there exists an additive $(G,k,\lambda)$-SDF with $G$ a zero-sum group of order $k$ 
and assume that $k$ has a prime divisor not dividing $\lambda$. 
Then there exists a $G$-super-regular $2$-$(v,k,1)$ design for infinitely many values of $v$.
\end{lem}
\begin{proof}
Let $\Sigma$ be a SDF as in the statement and let $p$ be a prime divisor of $k$ not dividing 
$\lambda$. Let $n$ be the order of $p$ 
in the group of units of $\Z_{\lambda}$, let $2^e$ be the largest power of $2$ dividing ${p^n-1\over \lambda}$, 
and set $\lambda_1=2^e\lambda$.
Clearly, $\underline{2^e}\Sigma$ is an additive $(G,k,\lambda_1)$-SDF. 
We have $p^n-1=2^e\lambda\mu$ with $\mu$ odd, 
hence $q^n\equiv\lambda_1+1$ (mod $2\lambda_1$). 
It easily follows, by induction on $i$, that $p^{ni}\equiv\lambda_1+1$ (mod $2\lambda_1$) for every odd $i$.
It is obvious that $|G|=k<2\lambda_1^{2k-5}$ and of course there are infinitely many odd values of $i$ for 
which $p^{ni}>(2k-3)^2\lambda_1^{4k-6}$. Hence, by Theorem \ref{additive version},
there exists an additive $(G\times\F_{p^{ni}},G\times\{0\},k,1)$-DF for each of these odd values of $i$.
The assertion then follows from Lemma \ref{additive(kq,k,k,1)}.
\end{proof}


\section{The main result}

For the proof of the main result we need one more ingredient, that is the notion of a {\it difference matrix}.

If $G$ is an additive group of order $v$, a $(v,k,\lambda)$ difference matrix in $G$ (or briefly a $(G,k,\lambda)$-DM)
is a $(k\times \lambda v)$-matrix with entries in $G$ such that the difference of any two distinct rows
contains every element of $G$ exactly $\lambda$ times.
For general background on difference matrices we refer to \cite{BJL,CD}.

We will say that a DM is {\it additive} if each of its columns is zero-sum.
An adaptation of an old construction for ordinary difference families by Jungnickel \cite{J} allows us to prove the following.

\begin{lem}\label{dieter}
If $\Sigma$ is an additive $(G,k,\lambda)$-SDF and $M$ is an additive $(H,k,\mu)$-DM, then there
exists an additive $(G\times H,k,\lambda\mu)$-SDF.
\end{lem}
\begin{proof}
Let $\Sigma$ be a $(G,k,\lambda)$-SDF and let $M=(m_{rc})$ be an $(H,k,\mu)$-DM.
For each block $B=\{b_1,\dots,b_k\}\in\Sigma$ and each column $M^c=(m_{1c},\dots,m_{kc})^T$
of $M$, consider the $k$-multiset $B\circ M^c$ defined as follows:
$$B\circ M^c=\{(b_1,m_{1c}),\dots,(b_k,m_{kc})\}.$$
It is straightforward to check that
$$\Sigma\circ M:=\{B\circ M^c \ | \ B\in{\cal F}; 1\leq c\leq \mu|H|\}$$
is a $(G\times H,k,\lambda\mu)$-SDF. It is clearly additive in the hypothesis that both
$\Sigma$ and $M$ are additive.
\end{proof}

In the proof of the following theorem we construct the crucial ingredient considered in Lemma \ref{crucial}.

\begin{thm}\label{goodSDF}
Let $k$ be a positive integer which is neither a prime power, nor singly even, nor of the form $2^n3$.
Then there exists an additive $(G,k,\lambda)$-SDF in a suitable zero-sum group of order $k$
with $\gcd(k,\lambda)=1$.
\end{thm}
\begin{proof}
Let $q$ be the largest odd prime power factor of $k$ and set $k=qr$.
The hypotheses on $k$ guarantee that $q$ is greater than 3. 
Now consider the $k$-multiset $A$ on $\F_q$ which is union
of $r$ copies of the $(q,q,q-1)$ Paley difference multiset of the first type: 
$$A=\underline{r}\{0\} \ \uplus \ \underline{2r}\F_q^\Box.$$

Let $\alpha: \F_q \longrightarrow \N$ be the map where $\alpha(x)$ is the multiplicity
of $x$ in $\Delta A$ for every $x\in \F_q$. We have 
$$\alpha(0)=r(r-1)+{q-1\over2}2r(2r-1)=(2q-1)r^2-qr.$$
Now let $x$ be an element of $\F_q^*$ and distinguish two cases according to whether
$q\equiv1$ or 3 (mod 4).

\underline{1st case}: $q\equiv1$ (mod 4).\quad
In this case it is well-known that $\F_q^\Box$ is a partial $(q,{q-1\over2},{q-5\over4},{q-1\over4})$ 
difference set\footnote{A $k$-subset $B$ of an additive group $G$ of order $v$ is a $(v, k, \lambda,\mu)$ {\it partial difference set} 
if $\Delta B=\underline{\lambda}(B\setminus\{0\}) \ \cup \ \underline{\mu}(G\setminus (B\cup\{0\})$.
If $\lambda=\mu$ then $B$ is a $(v,k,\lambda)$ {\it difference set}.}. 
If $x\in \F_q^\Box$, there are ${q-5\over4}$ representations of $x$ as a difference from $\F_q^\Box$.
Each of them has to be counted $(2r)^2$ times in the number of representations of $x$ as a difference from
$A$. The remaining representations of $x$ as a difference from $A$ are $x=x-0$ ($2r\cdot r$ times) and $x=0-(-x)$
($r\cdot2r$ times). Thus we have $\alpha(x)=(4r^2){q-5\over4}+2r^2+2r^2=(q-1)r^2$.

If $x\in \F_q^{\not\Box}$, there are ${q-1\over4}$ representations of $x$ as a difference from $\F_q^\Box$.
Each of them has to be counted $(2r)^2$ times in the number of representations of $x$ as a difference from
$A$. There is no other representation of $x$ as a difference from $A$. 
Hence we have $\alpha(x)=(4r^2){q-1\over4}=(q-1)r^2$.

\underline{2nd case}: $q\equiv3$ (mod 4).\quad
Here, $\F_q^\Box$ is a $(q,{q-1\over2},{q-3\over4})$ difference set. 
Every $x\in \F_q^*$ admits precisely ${q-3\over4}$ representations as a difference from $\F_q^\Box$.
Each of them has to be counted $(2r)^2$ times in the number of representations of $x$ as a difference from
$A$. The remaining representations of $x$ as a difference from $A$ are $x=x-0$ ($2r\cdot r$ times) if $x$
is a square, or $x=0-(-x)$ ($r\cdot2r$ times) if $x$ is not a square. 
Thus, for every $x\in \F_q^*$ we have $\alpha(x)=(4r^2){q-3\over4}+2r^2=(q-1)r^2$.

In summary, we have:
\begin{equation}\label{alpha}
\alpha(0)=(2q-1)r^2-qr\quad{\rm and}\quad\alpha(x)=(q-1)r^2 \ \forall x\in\F_q^*
\end{equation}
Now let $B=\underline{r}\F_q$ be the $k$-multiset which is union of $r$ copies of $\F_q$ and let
$\beta: \F_q \longrightarrow \N$ be the map of multiplicities of $\Delta B$. It is quite evident that we have:
\begin{equation}\label{beta}
\beta(0)=qr(r-1) \quad{\rm and}\quad  \beta(x)=qr^2 \ \forall x\in\F_q^*
\end{equation}
We claim that
$$\Sigma=\{A,\underbrace{B,\dots,B}_{r-1 \ {\rm times}}\}$$
is a $(q,k,(k-1)r^2)$-SDF in $\F_q$. Indeed, if $\sigma$ is 
the map of multiplicities of $\Delta \Sigma$, in view of (\ref{alpha}) and (\ref{beta}) we have:

\medskip
$\sigma(0)=\alpha(0)+(r-1)\beta(0)=(2q-1)r^2-qr+qr(r-1)^2=(qr-1)r^2;$

\medskip
$\sigma(x)=(q-1)r^2+qr^2(r-1)=(qr-1)r^2\quad\forall x\in\F_q^*.$

\medskip
Considering that $\F_q^\Box$ is a zero-sum subset of $\F_q$ for $q>3$ (see Fact \ref{fact2}), 
the multiset $A$  is zero-sum. Also, considering that $\F_q$ is zero-sum, $B$ is zero-sum as well. 
We conclude that $\Sigma$ is additive.

The hypothesis that $k$ is not singly even implies that $r$ is also not singly even.
Hence we can take an abelian zero-sum group $H$ of order $r$. 
Let $M$ be the matrix whose columns are all possible zero-sum
$k$-tuples of elements of $H$ summing up to zero. 
Let $(i,j)$ be any pair of distinct elements of $\{1,\dots,k\}$ and let $h$ be any element of $H$.
The number of zero-sum $k$-tuples $(m_1,...,m_k)$ of elements of $H$ such that $m_i-m_j=h$
is equal to $r^{k-2}$. Indeed each of these $k$-tuples can be constructed as follows. Fix any element $\ell$ in 
$\{1,...,k\}\setminus\{i,j\}$, take $m_x$ arbitrarily for $x \notin\{i,\ell\}$, and then we are forced
to take $m_i=m_j+h$ and $m_\ell=-\sum_{x\neq \ell}m_x$.

The above means that there are exactly $r^{k-2}$ columns $(m_{1,c},\dots,m_{k,c})^T$ of $M$ 
such that $m_{i,c}-m_{j,c}=h$. Equivalently, the difference between the $i$th row and the $j$th row
of $M$ covers the element $h$ exactly $r^{k-2}$ times.
Thus, in view of the arbitrariness of $i$, $j$ and $h$,  $M$ is a $(r,k,r^{k-2})$ difference matrix.
Of course it is additive by construction.

Thus, applying Lemma \ref{dieter}, we can say that $\Sigma\circ M$ is an additive $(k,k,\lambda)$-SDF 
in $G:=\F_q\times H$ with $\lambda=(k-1)r^k$. Recall that $q$ is the largest odd prime power factor of $k$ so that
$q$ is coprime with both $k-1$ and $r={k\over q}$. 
Thus $\lambda$ is coprime with $k$ and the assertion follows.
\end{proof}

{\bf Proof of Theorem \ref{main}.}\quad
If $k$ is a prime power we have the super-regular $2$-$(k^n,k,1)$ designs associated
with $AG(n,k)$. The singly even values of $k$ are genuine exceptions in view of
Proposition \ref{k=4n+2}(iv). Finally, if $k$ is neither a prime power, nor singly even, nor of the form $2^n3$,
then the assertion follows from Theorem \ref{goodSDF} and Lemma \ref{crucial}.
\hfill$\Box$

\section{A huge number of points}

As already mentioned in the introduction the super-regular Steiner
2-designs obtainable by means of the main construction (Theorem \ref{goodSDF} combined
with Lemma \ref{crucial}) have a huge number of points. On the other hand, there are some 
hopes to find more handleable super-regular Steiner 2-designs.  
We discuss this for the first relevant value of $k$, that is $k=15$.

Let us examine, first, which is the smallest $v$ for which the main construction leads
to a non-trivial super-regular $2$-$(v,15,1)$ design.
Keeping the same notation as in Theorem \ref{goodSDF}, we have $q=5$, $r=3$ and 
$\Sigma\circ M$ is a $(15,15,\lambda)$-SDF in $\Z_3\times\Z_5\simeq\Z_{15}$
with $\lambda=14\cdot3^{15}$. Now proceed as in the proof of Lemma \ref{crucial} taking $p=5$.
The order of $5$ in $\Z_{\lambda}$ is $n=2\cdot3^{14}=9565938$
and the largest power of 2 in ${q^n-1\over\lambda}$ is 4. Thus  $\underline{4}(\Sigma\circ M)$
is a $(15,15,\lambda_1)$-SDF with $\lambda_1=4\lambda$ and we have $5^{ni}\equiv \lambda_1+1$ (mod $2\lambda_1$)
for every odd $i$. One can check that $5^n>(2k-3)^2\lambda_1^{4k-6}=27^2\cdot(56\cdot3^{15})^{54}$.
Hence we have an additive $(\Z_{15}\times\F_{5^{n}},\Z_{15}\times\{0\},15,1)$-DF.
In conclusion, the first $v$ for which the application of Lemma \ref{crucial} with the use of 
$\Sigma\circ M$ leads to a super-regular
$2$-$(v,15,1)$ design is $3\cdot5^{9565939}$.

On the other hand, in this specific case, we can find a much lower $v$ with the
use of another SDF.
Consider the following three 15-multisets on $\Z_{15}$
$$B=\{0\} \ \cup \ \underline{2}\{1,2,3,7,9,11,12\};$$
$$B'=\{0\} \ \cup \ \underline{2}\{1,3,4,5,7,12,13\};$$
$$B''=\{0\} \ \cup \ \underline{2}\{1,5,8,10,11,12,13\}.$$
It is straightforward to check that $\Sigma'=\{B,B',B''\}$ is an additive $(15,15,\lambda')$-SDF
with $\lambda'=42$.
Let us apply Lemma \ref{crucial} using $\Sigma'$ rather than $\Sigma\circ M$. The order of $q=5$ in 
$\Z_{\lambda'}$ is $n=6$ and the largest power of 2 in ${q^n-1\over\lambda'}$ is $4$.
Thus $\underline{4}\Sigma'$
is a $(15,15,\lambda'_1)$-SDF with $\lambda'_1=4\lambda'$ and we have 
$5^{6i}\equiv \lambda'_1+1$ (mod $2\lambda'_1$)
for every odd $i$. The first odd $i$ for which $5^{ni}>(2k-3)^2{\lambda'_1}^{4k-6}$ is $31$.
Hence, the first $v$ for which the use of $\Sigma'$ in Lemma \ref{crucial} gives 
a super-regular $2$-$(v,15,1)$ design is $3\cdot5^{187}$.

Now we show a more clever use of $\Sigma'$ which exploits its nice form (every base block is of
the form  $\{0\} \ \cup \ \underline{2}A$ with $A$ a 7-subset of $\Z_{15}\setminus\{0\}$). 
Let $q\equiv1$ (mod 42) be a prime power and lift the blocks of $\Sigma'$ to
three zero-sum 15-subsets of $\Z_{15}\times \F_q$ of the form
\small
$$\ell(B)=\{(0,0),(1,\pm\ell_1),(2,\pm\ell_2),(3,\pm\ell_3),(7,\pm\ell_4),(9,\pm\ell_5),(11,\pm\ell_6),(12,\pm\ell_7)\},$$
$$\ell(B')=\{(0,0),(1,\pm\ell'_1),(3,\pm\ell'_2),(4,\pm\ell'_3),(5,\pm\ell'_4),(7,\pm\ell'_5),(12,\pm\ell'_6),(13,\pm\ell'_7)\},$$
$$\ell(B'')=\{(0,0),(1,\pm\ell''_1),(5,\pm\ell''_2),(8,\pm\ell''_3),(10,\pm\ell''_4),(11,\pm\ell''_5),(12,\pm\ell''_6),(13,\pm\ell''_7)\},$$
\normalsize
where, to save space, we have written $(x,\pm y)$ to mean the two pairs $(x,y)$ and $(x,-y)$.
We have $\displaystyle\Delta\ell(B) \ \cup \ \Delta\ell(B') \ \cup \ \Delta\ell(B'')=\bigcup_{i=0}^{14}\{i\}\times\Delta_i$ with
$\Delta_i=\{1,-1\}\cdot\overline{\Delta}_i$ where each $\overline{\Delta}_i$ is a list of 21 elements of $\F_q$. For instance, it is readily
seen that $\overline{\Delta}_0=\{\ell_i, \ell'_i, \ell''_i \ | \ 1\leq i\leq 7\}$.

Assume that the above liftings are done in such a way that each $\overline{\Delta}_i$ is a complete
system of representatives for the cyclotomic classes of order 21. In this case we have $\Delta_i\cdot M=\F_q^*$ with $M$
a system of representatives for the cosets of $\{1,-1\}$ in $C^{21}$ and then, by Construction \ref{constr},
we get an additive $(\Z_{15}\times\F_{q},\Z_{15}\times\{0\},15,1)$-DF.
Reasoning as in the proof of Theorem \ref{SDF->DF}, one can see that the required liftings
certainly exist by Lemma \ref{BP} provided that $q>6^2\cdot21^{12}$. Now note that
we have $5^{6i}\equiv1$ (mod 42) for every $i\geq0$ and $5^{6i}>6^2\cdot21^{12}$ as soon as $i\geq5$.
Thus we have an additive $(\Z_{15}\times\F_{5^{30}},\Z_{15}\times\{0\},15,1)$-DF.
So the first $v$ for which this construction leads, theoretically, to a strictly additive
$2$-$(v,15,1)$ design is $3\cdot5^{31}$ that is dramatically smaller than the value obtained
before by applying the main construction ``with the blinkers".
Yet, it is still huge! We cannot exclude, however, that by means of a (probably heavy) computer work
one may realize a good lifting of $\Sigma'$ with $q=5^6$. In this case we should have 
a $2$-$(3\cdot5^7,15,1)$ design.

\section{Super-regular non-Steiner 2-designs}
As underlined in the introduction, the paper is focused on super-regular Steiner 2-designs since their construction appears to be challenging. 
Here we just sketch how the methods used in the previous sections allow to obtain super-regular non-Steiner 2-designs
much more easily and with a relatively ``small" number of points. In particular, without any
need of cyclotomy (that is the heaviest tool used) it is possible to show that every additive $(k,k,\lambda)$-SDF with $k$
not singly even gives rise to a super-regular $2$-$(kq,k,\lambda)$ 
design for any power $q>k$ of a prime divisor of $k$.

First, we need to recall the following well known fact.
\begin{prop}\label{rem2}
Let ${\cal F}$ be a $(v,k,k,\lambda)$-DF in $G$ relative to $H$, let $\cal C$ be the set of right cosets of $H$ in $G$, and set
$${\cal B}=\{B + g \ | \ B\in{\cal F}; g\in G\} \ \cup \ \underline{\lambda}{\cal C}.$$ 
Then $(G,{\cal B})$ is a $G$-regular $2$-$(v,k,\lambda)$ design.
\end{prop}
The above is contained in Remark \ref{rem} (r2) for $\lambda=1$ and 
produces a non-simple design for $\lambda>1$. 
\begin{lem}\label{additive(kq,k,k,lambda)}
If there exists an additive $(G\times\F_q,G\times\{0\},k,\lambda)$-DF with $G$ a zero-sum group of order $k$,
then there exists a super-regular $2$-$(kq,k,\lambda)$ design.
\end{lem}
\begin{proof}
The $(G\times\F_{q})$-regular $2$-$(kq,k,\lambda)$ design obtainable from $\cal F$ using Proposition \ref{rem2} is clearly additive. 
The assertion follows.
\end{proof}
\begin{thm}
If there exists an additive $(k,k,\lambda)$-SDF with $k$ not singly even, then there exists a super-regular $2$-$(kq,k,\lambda)$ design
for every power $q>k$ of a prime divisor of $k$.
\end{thm}
\begin{proof}
Let $\Sigma=\{B_1,\dots,B_s\}$ be an additive $(k,k,\lambda)$-SDF in $G$ and let $q$ be a prime
power as in the statement. Take a zero-sum $k$-subset $L=\{\ell_1,\dots,\ell_k\}$ of $\F_q$
whose existence is almost evident\footnote{It is also an immediate consequence of a formula giving
the precise number of $k$-subsets of $\F_q$ whose sum is an assigned $b\in\F_q$ (see Theorem 1.2 in \cite{LW} or, for
an easier proof, Theorem 1.1(3) in \cite{Pavone2}).}.
Lift each block $B_h=\{b_{h1},\dots,b_{hk}\}$ of $\Sigma$ to the subset $L_h=\{(b_{h1},\ell_{1}),\dots,(b_{hk},\ell_{k})\}$
of $G\times\F_q^*$. By definition of a strong difference family,
we have $\Delta\{L_1,...,L_h\}=\biguplus_{g\in G}\{g\}\times\Delta_g$ where each $\Delta_g$ is a $\lambda$-multiset on $\F_q^*$
so that we have
\begin{equation}\label{penultimate}
\Delta_g\cdot\F_q^*=\underline{\lambda}\F_q^*\quad\forall g\in G.
\end{equation}
Given $m\in\F_q^*$, denote by $L_h\circ m$ the subset of $G\times\F_q$ obtained from $L_h$ by 
multiplying the second coordinates of all its elements by $m$.
Taking (\ref{penultimate}) into account, it is easily seen that 
\begin{equation}\label{last}
{\cal F}=\{L_h\circ m \ | \ 1\leq h\leq s; m\in \F_q^*\}
\end{equation}
is a $(G\times\F_q,G\times\{0\},k,\lambda)$-DF. Also, we note that ${\cal F}$ is additive since $\Sigma$ is additive 
and $L$ is zero-sum. The assertion then follows from Lemma \ref{additive(kq,k,k,lambda)}.
\end{proof}
Applying the above theorem using the $(15,15,42)$-SDF given in the previous section, we find
a super-regular $2$-$(15q,15,42)$ design for every power $q$ of 3 or 5 not smaller than 25. 
Here, however, in view of the special form of the used $(15,15,42)$-SDF, one could see that if $L$
is chosen more carefully as in Section 9 and if in (\ref{last}) we make $m$ vary in a system of
representatives for the cosets of $\{1,-1\}$ in $\F_q^*$ rather than in the whole $\F_q^*$, we get an additive
a $(\Z_{15}\times\F_q,\Z_{15}\times\{0\},15,\lambda)$-DF with $\lambda=21$ rather than 42.
Thus we can say  that there exists a super-regular $2$-$(15q,15,21)$ design for every power $q$ of 3 or 5 not smaller than 25.  
In particular, using $q=25$, we can say that there exists a super-regular $2$-$(375,15,21)$ design.


\section{Open questions}

Our research leaves open several questions.
The most intriguing is probably the following.
\begin{itemize}
\item[(Q1)] Does there exist a strictly $G$-additive Steiner 2-design which is not $G$-regular?
\end{itemize}
Here are some other questions which naturally arise.
\begin{itemize}
\item[(Q2)] Do there exist strictly additive $2$-$(v,k,1)$ designs with $k$ singly even?
\item[(Q3)] Do there exist super-regular Steiner 2-designs with block size $k=2^n3\geq12$?
\end{itemize}
Finally, it would be desirable to solve the following problem.
\begin{itemize}
\item[(P)] Find an additive Steiner 2-design with a non-primepower block size
and a ``reasonably small" number of points.
\end{itemize}

\section*{Acknowledgements}
The authors wish to thank the anonymous referees for their careful reading and some helpful comments. 

This work has been performed under the auspices of the G.N.S.A.G.A. of the C.N.R. (National Research Council) of Italy.

The second author is supported in part by the Croatian Science Foundation
under the projects 9752 and 6732.

\end{document}